\documentclass[a4paper,11pt]{article}

\usepackage[plainpages=false]{hyperref}
\usepackage{amsfonts,latexsym,rawfonts,amsmath,amssymb,amsthm, mathrsfs, lscape}
\usepackage{verbatim}
\usepackage[all]{xy}
\usepackage{authblk}
\usepackage{color}

\usepackage{array, tabularx}

\usepackage{setspace}
\usepackage[top=1.5in, bottom=1.5in, left=1.3in, right=1.3in]{geometry}

\newtheorem{thm}{Theorem}[section]
\newtheorem{cor}[thm]{Corollary}
\newtheorem{lem}[thm]{Lemma}

\newtheorem{prop}[thm]{Proposition}

\newtheorem{rmk}[thm]{Remark}

\newtheorem{defi}[thm]{Definition}

\numberwithin{equation}{section}\def \N {\mathbb N}

\def \Q {\mathbb Q}
\def \R {\mathbb R}
\def \C {\mathbb C}
\def \Z {\mathbb Z}
\def \P {\mathbb P}

\begin{document}

\title{On cscK resolutions of conically singular cscK varieties}

\author[1]{Claudio Arezzo \thanks{arezzo@ictp.it}}
\author[2]{Cristiano Spotti \thanks{c.spotti@dpmms.cam.ac.uk}}

\affil[1]{ICTP, International Center for Theoretical Physics\\}
\affil[2]{DPMMS, University of Cambridge}

\renewcommand\Authands{ and }

\maketitle

\begin{abstract}  In this note we discuss the problem of resolving conically singular cscK varieties to construct smooth cscK manifolds, showing a glueing result for (some) crepant resolutions of cscK varieties with discrete automorphism groups.  
 \end{abstract}

\vspace{5 mm}


\section{Introduction}

In recent years the problem of constructing special Riemannian metrics via the so-called glueing techniques has been object of intense investigations: just to mention a couple of results which share similarities with the content of this paper,  we may recall the Joyce's construction \cite{Joyce} of metrics with holonomy $G_2$ and $Spin(7)$ and the results on the existence of constant scalar curvature  K\"ahler (cscK, for short)  metrics on complex blow-ups \cite{Arezzo-Pacard}.

The setting of this paper is quite similar to the above mentioned ``generalized Kummer's constructions'': namely, we start by taking some singular complex variety $X$ which admits a cscK metric, in a sense which we will make precise later, and we prove the existence of cscK metrics on a \emph{resolution} $\hat{X}$ of $X$. Compared to the previous results, the main new feature consists in the fact that we now allow the singularities to be not just of isolated orbifold type \cite{Arezzo-Pacard}, \cite{Rollin-Singer}, \cite{Arezzo-Lena-Mazzieri1}, \cite{Arezzo-Lena-Mazzieri2}, but, 
more generally, we consider isolated singularities modeled on Calabi-Yau (CY for short) cones of a certain type. In the special case of CY 3-folds, some results in this direction, obtained using $G_2$-techniques, were proved by Chan in \cite{Chan}.

We say that a variety $X$ admits a \emph{conically singular cscK metric} $\omega$ if the metric $\omega$ is smooth on the non-singular locus $X^{reg}$ of the variety and, roughly speaking, there exist local biholomorphisms  which identify a neighborhood of a singularity with a neighborhood of the apex in a CY cone $(Y,\omega_Y)$ in such a way that in these adapted  ``charts''  the metric $\omega$ is asymptotical to a  CY cone metric $\omega_{Y}$ at some polynomial rate $\mu$ with respect to the radial cone distance (see Section \ref{MD} for the precise definitions).

The existence  of asymptotically (AC) crepant CY resolutions $(Y,\omega_Y)$  of a given cone $(Y,\omega_Y)$ is a problem which has attracted great attention by many authors \cite{Goto}, \cite{VCov}, \cite{VCov2}, \cite{Conlon-Hein} being a generalization of the celebrated Kronheimer's construction of gravitational instantons \cite{Kronheimer}. Here we will essentially use the result of \cite{VCov}.
Under the existence of such CY resolutions of a normal variety $X$ with only discrete automorphisms, we show that we can perform a glueing of a rescaled model of the asymptotic CY resolution with the complementary of a deleted neighborhood of the singularity on the original variety $X$. By standard perturbation arguments,  we obtain a new variety $\hat{X}$ which then admits smooth cscK metrics. 
Our main Theorem is thus the following:

\begin{thm}\label{MT}
Let $(X,\omega)$ be a complex $n$-dimensional with $n\geq 3$ conically singular cscK variety  with discrete automorphism group. Assume that the cscK metric $\omega$ on $X^{reg}$ is asymptotic to a metric Calabi-Yau cone at rate $\mu_i>0$ near $p_i\in \mbox{Sing}(X)$ with $i=1,\dots, l$, and that the Calabi-Yau cone singularities admit asymptotically conical Calabi-Yau resolutions $(\hat{Y}_i,\omega_{\hat{Y}_i})$ of rate $-2n$ and $i\partial\bar\partial$-exact at infinity (compare Section \ref{MD} for the definitions). 

Then there exists a crepant resolution $\hat{X}$ of $X$ admitting a family of smooth cscK metrics $\hat{\omega}_{\lambda}$, for  a parameter $\lambda \in \R^{+}$ small enough. Moreover, the sign of the scalar curvature of the cscK metrics $\hat{\omega}_{\lambda}$ on the resolution is the same as the one of the metric $\omega$ and, finally, the cscK manifolds $(\hat{X}, \hat{\omega}_{{\lambda}})$ converge to the singular cscK space $(X,\omega)$ in the Gromov-Hausdorff topology as $\lambda$ tends to zero.
\end{thm}

We highlight the following simple Corollary which partially generalizes to any dimension some results of Chan in dimension three \cite{Chan} (e.g., Theorem $5.2$).

\begin{cor} \label{CYcor} Let $(X,\omega)$ be a conically singular Ricci-flat CY variety satisfying the hypothesis of the main Theorem \ref{MT}. Then on the crepant resolution $\hat X$ we have a family $\omega_\lambda$ of smooth Ricci flat CY metrics converging to the singular space $(X,\omega)$ in the Gromov-Hausdorff topology. 
\end{cor}

The organization of the paper is as follow. In Section \ref{MD} we give the precise definitions of what we mean to be asymptotically cscK and we make some remarks on these types of metrics. In the way, we show that the condition of having a \emph{global} K\"ahler metric on $X^{reg}$ which locally near the singularities coincides with the local model CY cone metrics is essentially (e.g., for algebraic varieties) always unobstructed. 

In Section \ref{G}, we begin by constructing a parameter family of  approximate cscK metrics on a crepant resolution. Then we perform the analysis to deform these pre-glued families of metrics to genuine cscK metrics. Due to the present of an abundant literature  on the subject, and to keep the presentation as clean as possible, we mainly emphasize the differences in the glueing process with previous similar constructions (e.g., the simplified presentation of \cite{Arezzo-Pacard} given in Chapter $8$ of \cite{Szekelyhidi}. Compare also \cite{Szekelyhihi2}).

Finally, in Section \ref{E}, we discuss some examples and possible generalizations of our main Theorem \ref{MT}.
We should mention here that, strictly speaking, we do not know at present any example of asymptotically conical cscK variety, but there are at least conjectural ones. In particular, we give examples of  singular $4$-folds satisfying all the needed requirements, \emph{except}, but crucially, the asymptotic expansion with respect to the natural model metrics near the singularities (Proposition \ref{Ex}).   Determining the precise asymptotic of a cscK metric, or even K\"ahler-Einstein, near a singularity turned out to be extremely delicate. Thanks to the work of \cite{EGZ} and \cite{BBEGZ}, we now have abundance of K\"ahler-Einstein metrics on $X^{reg}$, but the comparison between ``the metric and the analytic tangent cones'' has escaped our understanding up to now (it is tempting to think that the techniques developed by Donaldson and Sun in the very recent \cite{Donaldson-Sun}  can give light to the asymptotic decay metric problem, at least in specific cases). Nevertheless, we think that this ``theoretical'' glueing construction  carries some interest in its own (being a model for similar constructions which start from a genuine singular, i.e., with not orbifolds singularities, space)  and it may give additional motivation in the very important study of the metric behavior near the singularities of complex varieties equipped with ``canonical'' metrics.

\subsection*{Acknowledgements}

We would like to thanks Hans Joachim Hein for useful discussions related to this project. The authors have been partially supported by the FIRB Project ``\emph{Geometria Differenziale Complessa e Dinamica Olomorfa}''.  The second author has been supported also by the ANR grant ANR-10-BLAN 0105 and by the EPSRC grant EP/J002062/1.

\section{Main definitions}\label{MD}

We begin by giving the definition of our local models at the singularities, that is Calabi-Yau cones. For more information we refer for example to \cite{Conlon-Hein} and \cite{VCov}.

\begin{defi} We say that a normal (pointed) complex $n$-dimensional non-compact analytic variety $(Y,o)$ is a \emph{Calabi-Yau cone (CY cone)} if it exists a smooth K\"ahler form $\omega$ on $Y\setminus \{o\}$ such that:
\begin{itemize}
\item $Y^{\ast}:=Y\setminus \{o\}$ is isometric to a Riemannian cone $C(L)=(\R^+\times L, dr^2+r^2 g_L)$, with $L$ smooth manifold of real dimension $2n-1$ (the link), and $r$ distance function from the cone apex $\{o\}$, so that the K\"ahler form can be written as $\omega=i\partial\bar\partial r^2$.
\item there exists a nowhere vanishing section $\Omega$ of the canonical bundle $K_Y$, such that the Calabi-Yau equation
$$\omega^n=c(n)\,\Omega\wedge \bar\Omega$$
holds on $Y^\ast$.
\end{itemize}

\end{defi}

More generally, we can consider isolated singularities with $K_Y$ only $\Q$-Cartier, i.e., such that there exist an $m \in \N$ so that $K_Y^{m}$ is Cartier.

\begin{rmk}\label{scaling} The above definition implies that the real link $L$ is a Sasaki-Einstein manifold. If we denote by $J$ a complex structure tensor inducing the complex analytic structure on the cone, we have that the Reeb vector of the link is given by $\xi=J \left(r\frac{\partial}{\partial r}\right)$.

The Euler field $r\frac{\partial}{\partial r}$ is real holomorphic and induces a family of biholomorphisms (scalings) $\sigma_\lambda: Y\cong C(L)\longrightarrow Y$:
$$\sigma_{\lambda}: (r,p) \longmapsto (\lambda^{-1}r,p),$$
for all $\lambda \in \R^+$. Observe that $\sigma_{\lambda}^{\ast}(\omega)=\lambda^{-2}\omega$.

\end{rmk}

Let $X$ be a normal complex n-dimensional compact analytic variety with $\Q$-Cartier canonical divisor and with isolated singularities. Denote its singular set with $S:=\mbox{Sing}(X)=\{p_1,\dots,p_l\}$.  

\begin{defi} We say that a variety $X$ as above has \emph{singularities modelled on CY cones}, if the following holds:
\begin{itemize}
\item\label{s} for all $p_i\in S$ there exists a local biholomorphism $\varphi_i:V_i\subseteq Y_i \rightarrow X$, where $V_i$ is an open neighborhood of the apex $o_i$  on a CY cone $Y_i$ and $\varphi_i(o_i)=P_i$.
\end{itemize}
Let $\omega$ be a smooth K\"ahler form on $X^{reg}:= X \setminus S$.  We say that $\omega$ is \emph{conically singular of rate $\mu_i\in \R^+$ at $P_i \in S$} if it is possible to choose a local biholomorphism $\varphi_i$ as above which in addition  satisfies the following propriety: 
\begin{itemize}
\item  the $ i\partial\bar\partial$-equation $$\varphi_i^\ast (\omega)-\omega_{Y_i}= i\partial\bar\partial  f_i$$ holds on $V_i^\ast:= V_i\setminus \{o_i\}$, where $\omega_{Y_i}$ denotes a fixed Calabi-Yau cone metric on $Y_i$ and $f_i$ a function in $C^\infty(V_i^\ast;\R)$ which obeys the decay condition: for any $k\geq 0$
$$ | \nabla^k_{\omega_{Y_i}} f_i |_{\omega_{Y_i}}=\mathcal{O}(r_i^{\mu_i+2-k}),$$ with respect to the distance function from the cone apex $r_i(p)=d_{\omega_i}(p,o_i)\ll 1$. Here the covariant derivatives and the norms are computed using the CY metrics $\omega_{Y_i}$.
\end{itemize}
Finally, we say that $(X,\omega)$ is  a \emph{conically singular cscK variety} if, in addition to the previous proprieties, the K\"ahler form $\omega$  satisfies the constant scalar curvature (cscK) equation
$$Sc(\omega):=i\,tr_{\omega}\bar\partial\partial \log\, \omega^n=c\in\R$$ on the regular part $X^{reg}$.

\end{defi}

Some comments on the above definitions are needed. First, note that the definition of conically singular immediately implies that the  metric decay property  $$| \nabla^k_{g_{Y_i}}\left( \varphi_i^\ast (g)-g_{Y_i}\right)|_{g_{Y_i}}=\mathcal{O}(r_i^{\mu_i-k})$$  holds.

As we mentioned in the introduction, and even if not strictly necessary for our arguments, it is then natural to ask if the condition of having a  \emph{global} K\"ahler metrics with conical singularities on a variety $X$, with singularities modelled on CY cones, but not satisfying in general particular curvature condition, is always achievable or it imposes some extra conditions.  Based on a simple glueing argument, we show that such metrics can be constructed in many natural situations, including the case of algebraic varieties $X \subseteq \P^N$.

In order to state the Proposition, recall that a variety $X$ is called \emph{smoothly K\"ahler} if it admits a smooth K\"ahler metric $\omega$ (on its non-singular locus) which, locally near the singularities, is given by the restriction of a smooth K\"ahler form an embedding of a neighborhood of each point into a smooth K\"ahler manifold.  The prototypical, and most important, example of such metric is the case of the K\"ahler metric an algebraic variety $X \subseteq \P^N$ induced by the restriction of the Fubini-Study form on $\P^N$.

\begin{prop} \label{NoO} Let $(X,\omega)$ have singularities analytically, but not metrically, modeled on CY cones and assume the metric $\omega$ is smoothly K\"ahler. Then there exists a K\"ahler metric $\omega^{'}$ on $X^{reg}$ with metrically conical singularities which satisfies $[\omega^{'}]=[\omega]\in H^{1,1}(X^{reg};\R)$. 

In particular, this holds for algebraic varieties $X\subseteq \P^N$, having analytically CY cone singularities.\end{prop}
\begin{proof}
Assume for simplicity that the singular set consists of only one point. By hypothesis we have a biholomorphism $\phi: V \subseteq Y \rightarrow U \subseteq X$ and an embedding $i:U \hookrightarrow Z$, where $(Z,\eta)$ is a smooth K\"ahler manifold and $\omega_{|U^\ast}=i^\ast \eta$. By making an analytic change of coordinate in $Z$,  we may assume without lost in the generality that $Z$ is an open ball in $\C^N$ centred at the origin, $i \circ \phi (o)= 0$  and $\eta= i\partial\bar\partial \left(|z|^2 +\mathcal{O}(|z|^4) \right)$.

Let $\chi:\R \rightarrow \R$ be a standard cut-off function equal to zero for $t\leq 1$ and identically one for $t\geq 2$. Define for $\delta_1 << 1$ the function   $ \chi_{\delta_1}(z):= \chi \left(\frac{|z|}{\delta_1}\right)$. Then, for $\delta_1$ sufficiently small, the closed $(1,1)$-form $\eta_{\delta_1}:= i\partial\bar\partial\left(|z|^2+ \chi_{\delta_1}(z) \mathcal{O}(|z|^4) \right)$ is K\"ahler and is equal to the flat metric on a possibly arbitrary small neighborhood of $\C^n$. Observe that for $|z|\geq 2 \delta_1$ the restriction of this metric agrees with $\omega$. Hence, abusing of notation, we may think of  $\eta_{\delta_1}$ to define a K\"ahler metric on all $X^{reg}$.

For $\delta_2 << 1$ consider a  function $\tau_{\delta_2}:\R \rightarrow \R$ identically equal to zero for $t\leq \delta_2^2$, equal to the identity for $t\geq 4 \delta_2^2$  and which satisfies $\tau^{'}_{\delta_2}(t),\tau^{''}_{\delta_2}(t)\geq 0$ for all $t$. Then take $2\delta_2 < \delta_1$ and define a closed $(1,1)$-form on $X^{reg}$
$$
 \tilde{\eta} = \left\{
  \begin{array}{l l}
  i\partial\bar\partial(\tau_{\delta_2}(|z|^2))_{|U}  & \quad \text{on }  |z|\leq 2\delta_2 \\
 \eta_{\delta_1} & \quad \text{on }  |z|> 2\delta_2

  \end{array} \right.$$
It is easy to check that $i\partial\bar\partial(\tau_{\delta_2}(|z|^2))(\xi)=\tau^{''}_{\delta_2}|(\xi,z)|^2+\tau^{'}_{\delta_2}|\xi|^2\geq 0$, where $\xi \in \C^n$ and the bracket is the usual hermitian product. Hence $\tilde{\eta}$ defines a non-negative closed form, identically zero on a small neighborhood of the singularity, such that $[\tilde{\eta}]=[\omega]$.

Now consider the cone CY metric $i\partial\bar\partial r^2$ on the cone Y. For $\delta_3<<1$ define a closed $(1,1)$ form on $U$ to be 
$\phi_\ast\left(i\partial\bar\partial((1-\chi(r/\delta_3)r^2)\right)$, and observe that on the compact region $\{\delta_3\leq r\leq 2\delta_3\}$ we do not have control on the sign of $\omega$.
Since all $\delta_1,\delta_2,\delta_3$ can be taken arbitrary small, we may assume $\phi(\{r\leq \delta_3\})\supseteq X\cap \{|z|\leq 2\delta_1\}$.

Then, abusing of notation we define on $X^{reg}$ the form,

$$\omega^{'}:=\epsilon^2 \phi_\ast\left(i\partial\bar\partial((1-\chi(r/\delta_3)r^2)\right) + \tilde{\eta}.$$
For $\epsilon<<1 $ the above form is positive and satisfies the desired properties. In fact, on   $\phi(\{r\leq \delta_3\})$ we have that $\omega^{'}$ is exactly conical near the singularity (take the scaling biholomorphism $\sigma_{\epsilon^{-1}}$)) and $\omega^{'}\geq \epsilon^2 i\partial\bar\partial r^2>0$. On the (compact) collar region $\phi\left(\{\delta_3\leq r\leq 2\delta_3\}\right)$ we can take $\epsilon$ so small that the strict positivity of $\tilde{\eta}=\omega$ dominates the possible negative contribution coming from the  cut-off. On the remaining region $\omega^{'}$ agrees with the starting form $\omega$. Obviously $\omega^{'}$ is in the same cohomology class of $\omega$. 
\end{proof}

The above Proposition \ref{NoO} naturally gives always a background conical metric in the relevant cohomology class. Thus it is tempting to try to prove \emph{``conical Calabi-Yau Theorems''} by starting with such background metrics and by deforming them via the classical continuity path to a K\"ahler-Einstein one  (for non-positive first Chern class) while keeping some control on the asymptotic behavior near the singularities. However, as it is well-known to experts in the field, there are severe difficulties in carrying out such program (for example the usual proof of the Yau's $C^2$-estimates doesn't generalize since the holomorphic bisectional curvature of any (non-flat) cone blows-up at the singularity, simply for scaling reasons). On the other hand, by the work of \cite{EGZ}, one can still equip a klt variety $X$ with $c_1(X)\leq0$ with some weak KE metric (which are, in particular, smooth on $X^{reg}$ and ``canonical'', i.e., unique). However, in this case any information about asymptotic expansions of the metric near 
the singularities are lost (it is only known that such metric admits an $L^\infty$ local potential).

Finally, we recall the definition of the local asymptotically conical CY models we will need to resolve the singularities.

\begin{defi}  We say that a CY cone $(Y,\omega)$ admits an \emph{asymptotically conical (AC) CY resolution $(\hat{Y},\hat{\omega})$ of rate $\nu\in \R$ and  $ i\partial\bar\partial$-exact at infinity}, if: 
\begin{itemize}
\item $\hat{Y}$ is a smooth manifold with trivial $m$-pluricanonical bundle, and there exists an holomorphic map (crepant resolution)  $\pi:\hat{Y}\rightarrow Y$ which is a biholomorphism away from $\pi^{-1}(o)$;
\item  the $ i\partial\bar\partial$-equation $$\pi_\ast (\hat{\omega})-\omega= i\partial\bar\partial \, g$$ holds on $Y\setminus B(R,o)$, where $B(R,o)$ is a ball on the CY cone of sufficiently big radius $R$ and $g$ a function in $C^\infty(Y\setminus B(R,o);\R)$ which obeys the decay condition: for any $k\geq 0$
$$ | \nabla^k_{\omega}\, g |_{\omega}=\mathcal{O}(r^{\nu+2-k})$$ with respect to the distance function from the cone apex $r(p)=d_{\omega}(p,o)\gg 1$. Here the covariant derivatives and the norms are computed using the CY metric cone metric. 
\end{itemize}
\end{defi}

While the above definition leaves the rate $\nu$ free, a result of Van Coevering shows that $\nu$ is \emph{always} equal to $-2n$ for compactly supported K\"ahler classes. In complex dimension $2$, these type of model are exactly given by the Kronheimer's ALE gravitational instantons \cite{Kronheimer}, i.e., asymptotically locally euclidean CY metrics on the minimal resolution of an isolated  quotient singularity of type $\C^2/\Gamma$, with $\Gamma\subseteq SU(2)$ finite, acting freely away from the origin. 
In higher dimension, the simplest source of examples comes from the well-known \emph{Calabi's Ansatz} \cite{Calabi}, which we now recall briefly (it will be applied later in the Proposition \ref{Ex}). Let $M^{n-1}$ be a $n-1$ dimensional K\"ahler-Einstein Fano manifold, and let $i\in \N$ the maximal number such that $ K_M=H^i \in Pic(M)$, for some (ample) line bundle $H$. Such $i$ is called the \emph{index} of the Fano manifold, and it is well-known to be always less than or equal to $n$ (which happens only for $\P^{n-1}$). Then, it is known that the $\Z_i$ quotient of the total space of the canonical bundle minus the zero section is a CY metric cone which admits a natural crepant resolution carrying an almost-explicit AC metric of rate $-2n$ and $i\partial\bar\partial$-exact at infinity.
Thus, a useful source of examples of  AC CY metrics with the desired properties is given by looking at some hypersurface singularities of the appropriate degree (we will make use of such examples later in Section \ref{E}): 
\begin{lem}\label{CA} Let $M^{n}\subseteq \C^{n+1}$, $n \geq 3$ be the ``generic'' smooth away from the origin hypersurface given by an homogeneous polynomial of degree $n$. Then $M^{n}$ admits a CY cone metric and $Bl_0(M)\cong Tot \left(\mathcal{O}_{\bar{M}}(-1)\right)$, where $ \bar{M}\subseteq \P^n$ is the natural complex link, admits an AC CY metric $i\partial\bar\partial$-exact at infinity.
\end{lem}
\begin{proof}
 It follows immediately by adjunction and Lefschetz's Theorem that $K_{\bar{M}}=\mathcal{O}(-1)_{|\bar{M}}$ and $Pic({\bar{M}})=\Z   K_{\bar{M}}$ if $n>3$ and it is classical that for cubic surfaces the canonical bundle is indivisible. It is known that Fermat type hypersufaces admits KE metric \cite{Tian} and then, since $Aut(\bar{M})$ is finite, the set of smooth hypersurfaces admiting KE metric is (at least) a not-empty Zariski open set \cite{Donaldson}, \cite{Odaka}.
\end{proof}

For more examples and discussions about these types of AC CY manifolds one can compare \cite{VCov}. Here we point out that there exists many examples of AC CY manifolds which \emph{are not} $i\partial \bar \partial$-exact at infinity (e.g. when the K\"ahler metric does not represent a compactly supported cohomology class). The simplest situation when this issue happens is given by the small resolutions of the ordinary double point $x^2+y^2+z^2+t^2=0$ (see discussion in Section \ref{E}).

\section{The glueing construction} \label{G}

In this section we prove our main Theorem \ref{MT}. The method we use is based on the \emph{smooth} (or orbifold) cscK blow-up construction given in  \cite{Arezzo-Pacard}, but following the simplified presentation given by Szekelyhidi in \cite{Szekelyhidi}. We begin with a ``pre-glueing'' construction.

\subsection{Pre-glueing}

On the complex manifold $\hat{X}_{{\lambda}}:= X \sharp_{\lambda}\sharp \hat{Y}_1  \dots \sharp_{\lambda}  \hat{Y}_l$, with ${\lambda} \in (0,1)$, naturally obtained by glueing $l$ asymptotically Calabi-Yau resolutions $\hat{Y}_i$ of order $-2n$ at the removed singularities of $X$, we define the following smooth real closed $(1,1)$-form: 

\begin{equation}\label{PG}
  \omega_{\lambda} = \left\{
  \begin{array}{l l}
    \omega & \quad \text{on } X\setminus B(2r_i(\lambda))\\
    i\partial\bar\partial \left(r_i^2+ \chi_i f_i+\lambda^2(1-\chi_i) g \circ \sigma_{\lambda}^i\right) & \quad \text{on }  B(2r_i(\lambda)) \setminus B(r_i(\lambda)) \\
    \lambda^2 \hat{\omega}_{Y_i} &  \quad \text{on }  \pi_i^{-1}(\sigma_{\lambda}(B(r_i(\lambda)))\subseteq \hat{Y}_i.
  \end{array} \right.
\end{equation}
where $r_i$ denotes the distance from the apex of the cone with respect to the cone metrics $i\partial\bar\partial r_i^2$ and $\hat{\omega}_{Y_i}$ is an AC Calabi-Yau metric of rate $\mu_i$ (we have omitted the local charts $\varphi_i$ in the above definition, for clarity of the notation), $\chi_i$ an usual cut-off function defined as $\chi_i(p):=\chi(r_i(p)/r_i(\lambda))$, with $\chi(t)$ a smooth increasing real valued function identically equal to zero for $t\leq1$ and identically equal to one for $t\geq 2$, and finally $\sigma_{\lambda}$ the scaling biholomorphisms induced by the Euler field as in Remark \ref{scaling}.

An elementary computation then shows the following (the actual choices for the cutting regions are important for later estimates):

\begin{lem} Let  $\omega_{{\lambda}}$ be  the $(1,1)$ form on $\hat{X}_{{\lambda}}$ as defined in  \ref{PG} above and let $$r_i(\lambda)= \lambda^{\frac{2n}{2n+\mu_i}}.$$ Then there exists $\lambda_0>0$ such that for all $\lambda \leq \lambda_0$ the forms $\omega_{{\lambda}}$ are positive definite (hence they define  K\"ahler metrics).
\end{lem}
\begin{proof}
For simplicity, assume that we only have one singular point. Away from the glueing region the statement is obvious. Moreover, thanks to the $-2n$ asympotic of the AC Calabi-Yau metric, it is immediate to check that, if the glueing region is choosen depending on $\lambda$ as $r(\lambda)=\lambda^\frac{2n}{2n+\mu}$, $\lambda^2 (g \circ \sigma_{\lambda})\sim r^{2+\mu}$ on the strip $[r(\lambda),2r(\lambda)]$ (this is the main reason for the specific choice of the region). Since the cut-off function scales homogeneously with the derivatives, on the glueing region, we find $|\omega_\lambda-\omega_Y|_{\omega_Y}=\mathcal{O}(r^\mu(\lambda))$. Hence, since $\mu>0$, the desired positivity is easily verified for $\lambda$ small enough.
\end{proof}

\begin{rmk}
\begin{itemize}
  \item  The fact that the (analytic) resolutions $\hat{X}_{{\lambda}}$ are K\"ahler is a consequence of the request of the cohomological triviality of the AC metrics $\hat{\omega}_{Y_i}$ at infinity. Without this assumption, k\"ahlerianity  of the (analytic) resolutions may fail due to some  \emph{global} phenomena, as the well-known case of some small resolutions of nodal threefolds. 
  \item  Note that the underlying holomorphic type of  the manifolds $\hat{X}_{{\lambda}}$ is the same. This simply follows by the natural scaling using the holomorphic action of the Euler field near the singularity and at the infinity of the non-compact model $\hat{Y}$ (note that, by Riemann's extension, the Euler field lifts to an holomorphic vector field on the resolution $\hat{Y}$). Thus we may sometime address to $\hat{X}_{{\lambda}}$ as to ``the'' complex manifold $\hat{X}$.
\end{itemize}
\end{rmk}

With the next Lemma we summarize the main ``cohomological'' properties of the pre-glued K\"ahler metrics. To state the Lemma we need to introduce some notation.
By applying the standard Mayer-Vietoris sequence in compact supported cohomology on $\hat{X}_{{\lambda}}$ with open sets $U\cong X^{reg}$, $V\cong \sqcup \hat{Y}_i$ and $U\cup V\cong \sqcup \left(L_i \times(0,1)\right)$, where the natural identifications with $\varphi_i$, $\pi_i$ and $\sigma_{\lambda}$ to be understood, we get a map
 $$h:H_c^2(U;\R)\oplus H_c^2(V;\R)\longrightarrow H^2_c(\hat{X}_{{\lambda}};\R)\cong H^2(\hat{X}_{{\lambda}};\R).$$
 Then
\begin{lem} \label{com}
The cohomolgy class of the preglued metric $\omega_{{\lambda}}$ is equal to 
 $$ [\omega_{{\lambda}}] = h\left(\sum_i \lambda^2 [\omega_{\hat{Y}_i}]_0 \oplus [\omega]_0\right) \in H^{1,1}(\hat{X}_{{\lambda}};\R)$$
 where with $[-]_0$ we denote the well-defined pre-image of the cohomology class in the compactly supported cohomology of $\hat{Y}_i$ and $X^{reg}$ respectively and with $h$ the natural map in cohomology coming from the Mayer-Vietoris sequence above recalled.

 Moreover, if there exists a cscK metric  $\omega_{{\lambda}}^\ast$ in $[ \omega_{{\lambda}} ]$, the value of the scalar curvature is given by:
$$ Sc( \omega_{{\lambda}}^\ast)=Sc(\omega)\frac{Vol(\omega)}{Vol( \omega_{{\lambda}})}=\frac{Sc(\omega)}{1+<\underline{c},{\lambda}^{2n}>},$$
for  ${\lambda}\leq \lambda_0$, where the signs in the vector $\underline{c}$ depends only on cohomological properties of the resolutions $\hat{Y}_i$ (compare the proof). In particular, the sign of the eventual constant scalar curvature in the class $[ \omega_{{\lambda}} ]$ must be equal to the sign of the scalar curvature of the original metric $\omega$.
\end{lem}
 
\begin{proof}
Since  $H^1( \sqcup  \left(L_i \times(0,1)\right); \R) \cong H^1 (\sqcup L_i; \R)=0$ (the links $L_i$ admits an Einstein metrics of positive scalar curvature) and since our model metrics are $i\partial \bar \partial$-exact both at infinity in $\hat{Y_i}$ and near the singularities of $X$, we get well-defined cohomolgy classes in the compact supported cohomology of the two glued pieces. Hence the first part of the statement follows.
 
 To show the second part, first notice that thanks to the imposed strong asymptotics (see \cite{VCov} and \cite{Conlon-Hein}), we have that $[\omega_{\hat{Y}_i}]_0$ must be cohomologus to the Poincar\'e dual of a suitable positive linear combination of the irreducible componens of exceptional set of the resolutions $\hat{Y}_i$ (which must be divisorial) that is $[\omega_{\hat{Y}_i}]_0= \sum_{i,j} \lambda^2  a_{i,j}[E_{i,j}]$ with $a_{i,j} \in \R$. By construction the resolution $\hat{X}_{{\lambda}}$ is $\Q$-Crepant, i.e.,
 $K_{\hat{X}_{{\lambda}}}=_{\Q} \pi^\ast K_X$. It follows that 
 $$c_1(\hat{X}_{{\lambda}}) \cup [\omega_{{\lambda}}]^{n-1}= c_1(\pi^\ast  K_X) \cup h \left( \sum_{i,j} \lambda^2  a_{i,j}[E_{i,j}] \oplus [\omega]_0\right)^{n-1}=c_1(K_X) \cup [\omega]^{n-1}.$$ 
 Here the last $c_1(X)$ means the Ricci curvature of $\omega$ on $X^{reg}$ (the integral makes sense thanks to the imposed asymptotics of $\omega$).
 Hence the eventual constant scalar curvature must be equal to
 $$n Sc( \omega_{{\lambda}}^\ast) {Vol( \omega_{{\lambda}})}=  c_1(\hat{X}_{{\lambda}}) \cup [\omega_{{\lambda}}]^{n-1}=c_1(K_X) \cup [\omega]^{n-1}= n Sc(\omega)Vol(\omega),$$
 Finally it is sufficient to expand the volume in powers of $\lambda$, i.e.,
 $$n! Vol( \omega_{{\lambda}})= \left( \sum_{i,j} \lambda^2  a_{i,j}[E_{i,j}] \oplus [\omega]_0\right)^{n}=[\omega]^n+(\sum_{i,j} \lambda^2  a_{i,j}[E_{i,j}])^n.$$
 \end{proof}

 \begin{rmk}
  The above Lemma is nothing else but a generalization of the well-known formula for the relevant cohomology classes and values of the scalar curvatures in the standard blow-up at smooth point as in \cite{Arezzo-Pacard}. 
 \end{rmk}

\subsection{Glueing analysis}

From now, for sake of notational simplicity, we assume that $Sing(X)=\{p\}$.

\subsubsection{Linear analysis}

Let us denote with $\Delta_C$ the Laplace operator for a $2n$-dimensional Riemannian cone $C(L)=(\R^+\times L, dr^2+r^2 g_L)$, with $L$ smooth and compact. It is well-known that in radial coordinates the Laplacian takes the form
$$\Delta_C=\frac{\partial^2}{\partial r^2} + \frac{(2n-1)}{r} \frac{\partial}{\partial r}+\frac{1}{r^2} \Delta_L,$$
where $\Delta_L$ is the Laplace operator on the compact link $L$
\begin{defi} $\xi \in \R$ is called \emph{indicial root} for $\Delta_C$ (resp. $\Delta_C^2$) if there exists a non identically zero smooth function $f\in C^{\infty}(L;\R)$ such that $\Delta_C(r^\xi f)=\mathcal{O}(r^{\xi-2+\eta})$, for some positive $\eta$ (resp. $\Delta_C^2(r^\xi f)=\mathcal{O}(r^{\xi-4+\eta}))$ for $r<<1$.

We denote with $\mbox{Ind}(\Delta_C) \subseteq \R$ (resp. $\mbox{Ind}(\Delta_C^2) \subseteq \R$) the set of indicial roots.
\end{defi}

Note that the definition seems slightly different from the standard one given in \cite{Arezzo-Pacard}. However, as one can easily see on the cone, the two definition actually agrees. In our case it seems more convenient to work with the above definition.

The next Proposition surveys some known relevant properties of the indicial roots for the bi-Laplacian. The proof is standard (see e.g. Theorem 3.2 in \cite{Conlon-Hein} and references therein). 

\begin{prop}Using previous notation, $\mbox{Ind}(\Delta_C^2) \subseteq \mbox{Ind}(\Delta_C)\cup\{\mbox{Ind}(\Delta_C)+2\}.$
Moreover, if $C(L)$ is not isometric to the flat space, there exists a positive $\epsilon=\epsilon(\lambda_1(L))$, depending only on the first eigenvalue of the Laplacian $\lambda_1(L)$ on the link, such that
$$\mbox{Ind}(\Delta_C^2) \cap (3-2n-\epsilon, 1+\epsilon)=\{4-2n,0\}.$$
\end{prop}

\begin{rmk}If $C(L)=\C^n/\Gamma$, with $\Gamma$ non-trivial, $$\mbox{Ind}(\Delta_{\C/\Gamma}^2) \cap [2-2n,2]=\{2-2n, 4-2n,0,2\}.$$
\end{rmk}

We now compute the indicial roots of the Laplacian of an asympotically conical metric $\omega$ of rate $\mu$ near the apex.

\begin{prop}\label{Lap}
 Let $\omega$ be an asympotically conical metric $\omega$ of rate $\mu$. Then take coordinates on the cone $Y\cong \R^{+}\times  C(L)$ of the form $(r,\underline{x})$, where $\underline{x}$ are some normal coordinates at $p$ on the link $L$ and $r$ is the distance function from the apex with respect to the model CY metric $\omega_Y$. Then $\Delta_{\omega}$ admits the following smooth expansion:
 \begin{eqnarray*}
  \Delta_{\omega} &= &(1+\mathcal{O}(r^{\mu}))\partial^2_{rr}+\sum_i \mathcal{O}_i(r^{\mu-1})\partial^2_{rx_i}+\sum_{i,j}\left(\frac{\delta_{ij}}{r^2}+\mathcal{O}_{i,j}(r^{\mu-2})\right)\partial^2_{x_ix_j}+ \\ 
 & & +\left(\frac{2n-1}{r} +\mathcal{O}(r^{\mu-1}) \right)\partial_r+\sum_i \mathcal{O}_i(r^{\mu-2})\partial_{x_i}.
 \end{eqnarray*}
Thus $(\Delta_{\omega}-\Delta_{\omega_{Y}})(r^\alpha f)=\mathcal{O}(r^{\mu+\alpha-2})$. In particular, $\mbox{Ind}(\Delta_{\omega})=\mbox{Ind}(\Delta_{\omega_Y})$.

\end{prop}
\begin{proof}
 Write the metric $g$ in ``polar normal coordinates'' as $$g=\left(\begin{matrix} 1+\mathcal{O}(r^\mu) & \mathcal{O}(r^{\mu+1}) \\ \mathcal{O}(r^{\mu+1}) & r^2(1+\mathcal{O}(r^\mu))\end{matrix}\right).$$
 Then the estimate follows by a long but straightforward computation using the formula of the Laplacian $\Delta=\sqrt{g}\partial_i(\sqrt{g}g^{ij}\partial_j)$.
\end{proof}

It is well-known (e.g., \cite{LeBrun-Simanca}) that the linearized operator of the scalar curvature is equal to:
$$\mathbb{L}_\omega(\varphi)= \Delta_{\omega}^2(\phi) + <Ric(\omega),i\partial\bar\partial \varphi>.$$ 
Then, thanks to the above estimates, we have:

\begin{prop} Let $\mathbb{L}_{\omega}$ be the linearized operator of the scalar curvature. Then $\mbox{Ind}(\mathbb{L}_{\omega})=\mbox{Ind}(\Delta_{\omega_{Y}}^2)$. In particular, there are no-indicial roots in $(4-2n,0)$.
 
\end{prop}
\begin{proof}
 It is sufficient to observe that $Ric(\omega)=\mathcal{O}(r^{\mu-2})$. Hence $( \mathbb{L}_\omega-\Delta_{\omega}^2 )(r^\alpha \varphi)=\mathcal{O}(r^{\alpha-4+\epsilon})$ for some $\epsilon \leq \mu$.
\end{proof}

A completely analogous discussion can be made for the indicial roots for the AC CY metrics on $\hat{Y}$.

Now we introduce the relevant function spaces. We start on the model $\hat{Y}$. Denote with $C^{k,\alpha}_\delta(\hat{Y};\R)$ the space of $C^{k,\alpha}_{loc}$ regular real functions $\varphi$ for which the following norm is finite:
\begin{equation*}\begin{split}
 |\varphi|_{C^{k,\alpha}_\delta}:= &\sum_{i=1}^k \sup_{\hat{Y}}|\rho^{j-\delta} \nabla_{\omega_{\hat{Y}}}^j \varphi | + \\ & + \sup_{p\neq q\in \hat{Y}, d(p,q) \leq inj}\left( \min\{\rho(p),\rho(q)\}^{-\delta+k+\alpha} \frac{|\nabla^k \varphi(p) - \nabla^k \varphi(q)|}{d(p,q)^\alpha} \right), 
\end{split}\end{equation*}
where $\rho$ is a weight function equal to $r_Y$ near infinity (and in the region of glueing) and decreasing to a non zero positive constant in the interior. Note that changing the weight function or the conical metrics by objects of the same behavior does not change the functional spaces and the norm remain equivalent.  

It follows by standard arguments (see e.g. \cite{Pacini} and references therein)  that if $\delta$ is not an indicial roots, $\mathbb{L}$ is a Fredholm operator of index zero. Hence $\dim\, Ker ({\mathbb{L}_\delta})=\dim\, Im (\mathbb{L}_{4-2n-\delta})$.

\begin{lem} \label{LA} If $\delta \in (4-2n,0)$ then
$$ \mathbb{L}_{\omega_{\hat {Y}}}: \mathcal{C}^{4,\alpha}_{\delta}(\hat{Y};\R) \longrightarrow  \mathcal{C}^{0,\alpha}_{\delta-4}(\hat{Y};\R), $$
is an isomorphism.
\end{lem}
\begin{proof}
 Thanks to the above remarks, it is sufficient to check that $\mathbb{L}_{\omega_{\hat {Y}}}$ has no non-trivial kernel. But this follows by applying twice the Liouville's Theorem, since $\mathbb{L}_{\omega_{\hat {Y}}}=\Delta_{\omega_{\hat {Y}}}^2$ being $\omega_{\hat {Y}}$ Ricci-flat by hypothesis. 
\end{proof}

Similarly one can define weighted space $C^{k,\alpha}_\delta(X^{reg};\R)$ on the non compact manifold $X^{reg}$ (now with weight function equal to the distance from the singularity near the singular locus and constant away from it). Then one can prove the following:

\begin{lem} \label{LB} If $\delta \in (4-2n,0)$ and $Aut(X)$ is discrete and let 
$$ \mathbb{L}_{\omega}: \mathcal{C}^{4,\alpha}_{\delta}(X^{reg};\R) \longrightarrow  \mathcal{C}^{0,\alpha}_{\delta-4}(X^{reg};\R). $$
Then $ker(\mathbb{L}_{\omega})\cong coker(\mathbb{L}_{\omega})$ consists of constant functions only. 
\end{lem}
\begin{proof}

As before it suffices to study the kernel. By the hypothesis on the weights and the indicial-roots properties, it follows that any $\varphi \in ker(\mathbb{L}_{\omega})$ must in fact be in $\mathcal{C}^{4,\alpha}_{0}(X^{reg};\R)$. Thus, by integration by parts around a singularity $p$,
$$0=\int_{X\setminus B_\omega(p,r)} \varphi \mathbb{L}_\omega(\varphi) \omega^n= \int_{X\setminus B_\omega(p,r)} |\bar\partial(\bar \partial \phi)^{\sharp_\omega}|^2 \omega^n +\mathcal{O}(r^{2n-4})$$
Here we used the well-known fact (see for example \cite{Arezzo-Pacard}) that for a cscK metric the linearized operator $\mathbb{L}_\omega$ is equal to the Lichnerowitz operator $D^\ast D$ where $D(\varphi):= \bar\partial(\bar \partial \varphi )^{\sharp_\omega}$ and $D^\ast$ is the (formal) $L^2$-adjoint. It follows that the $(1,0)$ vector field $(\bar \partial \phi)^{\sharp_\omega}$  must be holomorphic on $X^{reg}$.
Since the variety $X$ is normal, it follows \cite{Fischer} that one has a natural injection
$$i: H^0(X^{reg}, TX^{reg}) \hookrightarrow H^0(X, \mathcal{H}om (\Omega^1_X,\mathcal{O}_X)).$$
Since, by hypothesis, $Aut(X)$ is assumed to be discrete, the latter cohomology group vanishes and thus $(\bar \partial \phi)^{\sharp_\omega}=0$. Being the fuction $\varphi$ real, it follows that $\varphi$ must be constant. 

\end{proof}

\begin{rmk}\label{GTrick}
 
In order address the issue of the kernel, we follow the trick of Szekelyhidi in \cite{Szekelyhidi} of considering the following modification of the operator $\mathbb{L}_\omega$. Namely we consider the operator $\tilde{\mathbb{L}}_\omega(\varphi):= \mathbb{L}_\omega (\varphi) + \varphi(q)$, for $q\in X^{reg}$, taken, say, at distance $1$ from the singular set. It is then obvious that such operator $\tilde{\mathbb{L}}_\omega$ is now an \emph{isomorphism}.
 
\end{rmk}

Finally we introduce the weight spaces $C^{k,\alpha}_{\delta}$ on $\hat{X}$ with respect to the preglued metric $\omega_\lambda$. These spaces are given by $C^{k,\alpha}$ functions where the norm we consider is defined in a similar way as above but with weight function $\rho$, keeping in mind the identification given in the pre-glueing construction, is equal to the constant $\lambda$ if $r<\lambda$ (thought in the resolution!), it is an increasing function from $\lambda$ and $2\lambda$ which then becomes equal to the distance $r$ until $r=\frac{1}{2}$ and finally increases to a constant.
Another way to think about such norm is via glueing of rescaled version of the weighted norms on the glueing pieces (compare \cite{Szekelyhidi}). That is
$$ |\varphi|_{C^{k,\alpha}_\delta(\hat{X})} \simeq |\chi \varphi|_{C^{k,\alpha}_\delta(X^{reg},\omega)} + \lambda^{-\delta} |(1-\chi) \varphi |_{C^{k,\alpha}_\delta(\hat{Y}, \lambda^2\omega_{\hat{Y}})},$$
where $\chi$ is the cut-off function introduced in the pre-glueing. This equivalence just follows by the scaling properties of the identification  $\sigma_\lambda$.

The following useful estimates are elementary (compare Lemma 7.8 in \cite{Szekelyhidi}), and similar to Proposition \ref{Lap}. There exist constants $c_0,C$ such that if $|\varphi|_{C^{4,\alpha}_2} \leq c_0$ then $\omega_\varphi:= \omega_\lambda+i\partial\bar\partial \varphi$ is positive and:
\begin{equation}\label{Est}\begin{split}
 & |g_{\varphi}-g_\lambda|_{C^{2,\alpha}_{\delta-2}}+|Rm_\varphi-Rm_\lambda|_{C^{0,\alpha}_{\delta-4}} \leq C  |\varphi|_{C^{4,\alpha}_{\delta}} \\ &
  |\mathbb{L}_\varphi-\mathbb{L}_\lambda|_{B(C^{4,\alpha}_\delta,C^{0,\alpha}_{\delta-4})} \leq C |\varphi|_{C^{4,\alpha}_{2}}
\end{split}\end{equation}

Then let $\tilde{\mathbb{L}}_\lambda$ be the linear operator equal to $\mathbb{L}_\lambda+ \delta_q$. Then, combining all the results and definitions above we have the following uniform estimate for the inverse of $\tilde{\mathbb{L}}_\lambda$. Since the proof -given the invertibility of the corresponding operators on the pieces of the glueing construction-  is identical to Theorem $7.9$ in \cite{Szekelyhidi}, we will simply sketch the main ideas for readers' convenience.

\begin{prop}[Theorem 7.9 \cite{Szekelyhidi}]  \label{UI} Let $\delta \in (4-2n,0)$ and let $\omega_\lambda$ be the family pre-glued K\"ahler metrics. Then for $\lambda$ sufficiently small 
 $$ \tilde{\mathbb{L}}_{\lambda}: \mathcal{C}^{4,\alpha}_{\delta}(\hat{X}_\lambda;\R) \longrightarrow  \mathcal{C}^{0,\alpha}_{\delta-4}(\hat{X}_\lambda;\R) $$
 is an isomorphism, with $| \tilde{\mathbb{L}}_{\lambda}^{-1} |< C$ where $C$ is independent of $\lambda$.
\end{prop}

\begin{proof}
 The basic idea consists in constructing an approximate right inverse. Using the cut off functions $\chi$ in the definition of the preglued metric $\omega_\lambda$, we define the bounded operator
 $\mathbb{R}_\lambda:C^{0,\alpha}_{\delta-4}(\hat{X})\rightarrow C^{4,\alpha}_\delta(\hat{X})$, given by
 $$\mathbb{R}_\lambda(\varphi):=\beta_1\tilde{\mathbb{L}}^{-1}_\omega(\chi \varphi)+ \lambda^{-4} \beta_2 \mathbb{L}_{\eta}^{-1}((1-\chi)\varphi).$$
 where the inverses of the operators on the two pieces exists thanks to Lemmas \ref{LA} and \ref{LB}, and the function $\beta_i$ are cut off functions, identically equal to one on the the glueing region, satisfying the property $|\nabla \beta_1|_{C^{3,\alpha}_{-1}} \leq \frac{C}{\log(\lambda)}$ (and similarly for $\beta_2$). For example, take $\beta_1=\beta\left(\log(r)/\log(\lambda)\right)$ with $\beta(x)$ cut-off function equal to one for $x< \frac{2n}{2n+\mu}$ and equal to zero for $x> b > \frac{2n}{2n+\mu}$ with $b<1$. 
 
 With these choices is then possible to check that $|\tilde{\mathbb{L}}_\lambda \circ \tilde{\mathbb{R}}_\lambda-1|_{B(C^{0,\alpha}_{\delta-4},C^{0,\alpha}_{\delta-4})}  \rightarrow 0,$ as $\lambda\rightarrow 0$. Thus $\tilde{\mathbb{L}}_\lambda \circ \tilde{\mathbb{R}}_\lambda$ is invertible with bounded (independent of $\lambda$) norm. The uniformly bounded inverse is then given by $\tilde{\mathbb{L}}_\lambda^{-1}:=\tilde{\mathbb{R}}_\lambda \circ (\tilde{\mathbb{L}}_\lambda \circ \tilde{\mathbb{R}}_\lambda)^{-1}$.
 
 \end{proof}

\subsubsection{Deformation to a genuine solution}

We begin this section by an estimate of the ``error'' term in the pre-glueing construction. 
The term we need to estimate is
$$\varepsilon(\lambda)=Sc(\omega_\lambda)-Sc(\omega),$$
where $Sc(\omega)$ is the constant value of the scalar curvature of $(X,\omega)$.
\begin{lem}\label{error} The error term can be estimated as
 $$|\varepsilon(\lambda)|_{C^{0,\alpha}_{\delta-4}(\hat{X})}= \mathcal{O}(\lambda^{\frac{2n(\mu+2-\delta)}{2n+\mu}})$$
\end{lem}
\begin{proof}
We divide the manifold $\hat{X}_\lambda$ in three parts: let $A_\lambda$ be the region  $ X\setminus B(2r(\lambda))$,  $B_\lambda=B(2r(\lambda)) \setminus B(r(\lambda))$ and $C_\lambda=\pi^{-1}(\sigma_{\lambda}(B(r(\lambda)))$

\emph{Estimate region} $A_\lambda$: since on this region the pre-glued metric coincides with $\omega$, the error term $\varepsilon(\lambda)$ vanishes identically.

\emph{Estimate region} $B_\lambda$ (Glueing region): The pre-glued metric satisfies: $\omega_\lambda=i \partial \bar \partial (r^2 + \psi)$ with $\psi$ smooth function satisfying $\psi=\mathcal{O}(r^{2+\mu})$, and similarly for the derivatives. Hence
$$|\psi|_{C^{4,\alpha}_\delta (B_\lambda)} = \mathcal{O}(r_\lambda^{\mu+2-\delta}).$$
Observe that, by the properties of the weighted norms, the above estimate implies $|\psi|_{C^{4,\alpha}_2(B_\lambda)}=\mathcal{O}(r_\lambda^{\mu}).$ Hence, denoting with $\omega_Y$ be the CY (in particular scalar flat) cone metric, the first estimate in \ref{Est} implies:
$$|Sc(\omega_\lambda)|_{C^{0,\alpha}_{\delta-4} (B_\lambda)}=|Sc(\omega_\lambda)-Sc(\omega_{Y})|_{C^{0,\alpha}_{\delta-4} (B_\lambda)}\leq C r_\lambda^{\mu+2-\delta}.$$ Thus, since $\mu \in (0,2]$:
$$|\varepsilon(\lambda)|_{C^{0,\alpha}_{\delta-4}(B_\lambda)}\leq C r^{\mu+2-\delta} +|Sc(\omega)|_{C^{0,\alpha}_{\delta-4}(B_\lambda)}\leq \tilde{C} r_{\lambda}^{\mu+2-\delta},$$

\emph{Estimate region} $C_\lambda$: Since $Sc(\omega_\lambda)=0$ and $\lambda^{\frac{2n}{2n+\mu}}=r_\lambda>> \lambda$,
$$|\varepsilon(\lambda)|_{C^{0,\alpha}_{\delta-4}(C_\lambda)}= |Sc(\omega)|_{C^{0,\alpha}_{\delta-4}(C_\lambda)}=  \mathcal{O}( r_\lambda^{4-\delta}).$$

To conclude the proof of the lemma, simply recall that by our choices $r_\lambda \sim  \lambda^{\frac{2n}{2n+\mu}}$.
\end{proof} 

To prove the  Theorem \ref{MT} it is now sufficient to apply an usual fixed point argument. We want to solve the equation $$Sc(\omega_\lambda + i\partial \bar \partial \varphi ) = Sc(\omega)+\varphi (q) \; \left( = Sc^{\ast}(\omega_\lambda)\in \R \right),$$  
for a fixed $q \in X^{reg} $ sufficiently away from the singularities. This is of course equivalent to find a fixed point of the map:
$\mathcal{F}_\lambda: C^{4,\alpha}_\delta \rightarrow C^{4,\alpha}_\delta $ defined by
$$\mathcal{F}_\lambda(\varphi):=\tilde{\mathbb{L}}_\lambda^{-1} \left( \varepsilon(\lambda) -\N_\lambda(\varphi) \right),$$
where $\tilde{\mathbb{L}}_\lambda^{-1}$ is the uniformly bounded inverse of the linearized operator constructed in Lemma \ref{UI} for $\delta \in (4-2n,0)$ and  $\N_\lambda$ collects the non-linearities. 

Now, if the pointwise norm $|i\partial \bar \partial \varphi_j|$ (and its derivatives) is sufficiently small compared to the CY model metric (which happens in particular if $|\varphi_j|_{C^{4,\alpha}_2}$ is less than or equal to a small $b>0$ depending on the model metric), it is a known fact  $\cite{Arezzo-Pacard}$ that the non-linearities are in fact controlled by their the quadratic part:
$$|\N_\lambda(\varphi_1)-\N_\lambda(\varphi_2)|_{C^{0,\alpha}_{\delta-4}} \leq C \left( |\varphi_1|_{C^{4,\alpha}_2}+ |\varphi_2|_{C^{4,\alpha}_2}\right) |\varphi_1-\varphi_2|_ {C^{4,\alpha}_\delta}.$$
Hence, working in a sufficiently small ball in the $C^{4,\alpha}_2$-topology, it follows by the uniform invertibility of $\tilde{\mathbb{L}}_\lambda$ in  $C^{4,\alpha}_\delta$, that $\mathcal{F}_\lambda$ is Lipschitz, with Lipschitz costant strictly less than $1$.

To apply the contraction principle, one needs  to work in the ball $|\phi|_{C^{4,\alpha}_\delta} \leq \tilde{b} \lambda^{2-\delta}$, where the constant $\tilde{b}$ is chosen so that 
$$|\varphi|_{C^{4,\alpha}_2}\leq \tilde{C} \lambda^{\delta-2} |\phi|_{C^{4,\alpha}_\delta}\leq \tilde{C} \tilde{b},$$ is so small that, following the argument above,  $\mathcal{F}_\lambda$ has Lipschitz constant less than one. Thus, in order to have that $\mathcal{F}_\lambda$ sends such $\tilde{b}\lambda^{2-\delta}$-ball into itself, we want
$$  |\mathcal{F}_\lambda(0)|_{C^{4,\alpha}_\delta} \leq  C_1 |\varepsilon(\lambda)|_{C^{0,\alpha}_{\delta-4}} << \lambda^{2-\delta}. $$
Applying the above Lemma \ref{error}, we see that this happens if  $$\frac{2n(\mu+2-\delta)}{2n+\mu}>2-\delta.$$ Thus, since $\mu$ is positive, it is sufficient to take $\delta>4-2n>2-2n$. In conclusion, we have constructed  constant scalar curvatures K\"ahler metric on $\hat{X}$ as long as $\lambda$ is taken to be sufficiently small. The actual value of the constant scalar curvature is then given by Lemma \ref{com}.

\section{Examples $\&$ Discussion} \label{E}

We begin this final section by showing what happens if, in our main Theorem \ref{MT}, the starting singular variety $(X,\omega)$ is a conicaly singular Ricci-flat CY variety (i.e., we assume in addition that $K_{X}$ is a trivial $\Q$-Cartier divisor), hence proving Corollary \ref{CYcor}.
Let $\omega_\lambda$ the cscK metrics constructed in Theorem \ref{MT}. Since the resolution $\hat{X}$ is assumed by hypothesis to be crepant, we know that its first Chern class $c_1(\hat X)$ must be trivial in real De Rham cohomology. Moreover, since $$0=c_1(\hat X)\cap [\omega_\lambda]^{n-1} =n S_\lambda Vol(\omega_\lambda),$$ we have that $S_\lambda\equiv 0$. But, as the pre-glueing shows, $\hat{X}$ is K\"ahler (hence the $i\partial\bar\partial$-lemma holds) and thus $Ric(\omega_\lambda)=i\partial\bar\partial \phi_\lambda$ globally. So, taking the trace of the previous equation,  we get that $\phi_\lambda$ must be constant (being harmonic), hence proving the Corollary.

Next we search for other situations where the main Theorem \ref{MT} should apply. We will exhibit a variety (actually many of them) satisfying all the hypothesis of the main Theorem \ref{MT} \emph{besides the asymptotic expansion near the singularities}, as discussed in the introduction. 

The natural source of examples is given by varieties in the non-collapsing part of the boundary of the compactification of the moduli space of K\"ahler-Einstein manifolds. The next explicit example is found by looking to some boundary component in the moduli spaces of projective complete intersections of sufficiently high degree. Basically, all we need to ensure is that the singularities of such boundary variety are isolated and satisfy the hypothesis we need.
Thus, let us consider the following complete intersection in $\P^6$:
$$ X_s:=\begin{cases} \sum_{i=0}^4 s_i x_i^4 + y^4+z^4=0 \\ \sum_{i=0}^4  x_i^4=0 \end{cases},$$
where $s_i \in \C$ are generic (in particular all different). Then: 
\begin{prop} \label{Ex} The 4-fold $X_t$ defined as above is a klt variety with ample canonical line bundle with exactly four singularities analytically equivalent to the cone over the Fermat's hyperquartic, that is with singularities locally given by the equation $\sum_{i=0}^4  x_i^4=0$ in $\C^5$. Moreover, it admits a weak KE metric of negative scalar curvature and the singularities admits AC CY crepant resolutions.
\end{prop} 
\begin{proof}
 The fact that the canonical divisor is an ample Cartier divisor follows immediatly by the adjunction formula. To see that the singularities are the ones claimed, note that, by the symmetries of the defining equations, it is sufficient to check if some singularities appear in the charts $y \neq 0$ and $x_0 \neq 0$. Suppose first $y \neq 0$ and  assume that $z\neq 0$. By looking to the Jacobian of the equations, it is easy to see that the only singularities are given by $[0:0:0:0:0:1,\sigma]$ with $\sigma^2=-1$. If $z=0$, the Jacobian has no maximal rank only if all but eventually one of the $x_i$ vanish (due to the genericity of the $s_i$s), and clearly there are no such points on the variety.
Similarly one can show that there are no singular points in the chart $x_0\neq 0$. Thus it remains to show that the singularities are of the desired analytic type. Take $y=1$ and define a new variable $t:=\sum_{i=0}^4 s_i x_i^4 + 1+z^4$. It is clear that such change of variables (the $x_i$s being fixed) defines a local analytic biholomorphism away from $z=0$ (thus in particular at the point of the singularities). Thus, locally analytically, the singularities are given in the $(t,x_i)$ coordinates as the transverse intersection of $t=0$ with the bundle of cones $\sum_{i=0}^4  x_i^4=0$, thus proving the first part of the claims. 

Since $X_s$ has canonical singularities, on $X_s$ there exists a unique weak KE metric of negative scalar curvature \cite{EGZ}. Finally observe that the singularities of $X_s$ admits AC CY resolution thanks to the Calabi's Ansatz (see Lemma \ref{CA} and subsequent remark).  
\end{proof}

\begin{rmk}
 \begin{itemize}
  \item We note that the compex analytic variety $\hat{X_s}$ constructed via ``pre-glueing'' in Proposition \ref{Ex} is the algebro-geometric blow-up  $Bl_{Sing(X_s)}(X_s)$ at the singular set of $X_s$. 
  \item $X_s$ is one of the simplest example of complete intersection varieties where the result applies. Unfortunately 3-folds with ordinary double points are excluded by the range of applicability of our main Theorem (since small resolutions admits non $i\partial \bar \partial$-exact metrics at infinity). However, it may be possible to find examples by taking suitable global $\Z_2$ quotients (with action free on $X^{reg}$ and which, locally near the singularities, agrees with the action required by the Calabi's Ansatz) of the ubiquitous varieties having only ordinary double points as singularities (sometimes called hyperconifolds in the physics literature). 
  \item To construct manifolds with \emph{positive} constant scalar curvature, one could start by K-polystable Fano varieties in the K-moduli compactification. However showing existence of a weak KE metrics, for example using some alpha-invariant computation \cite{BBEGZ}, may be very complicated even for Fano complete intersections.
 \end{itemize}
\end{rmk}

Let us end by observing the role played by the CY condition on $\hat{Y}$: compared to the more general (and certainly natural) scalar-flat case, the only place where $\mbox{Ric}=0$ was really used in the proof has essentially been only Lemma $\ref{LA}$. We believe this would still be true for scalar-flat models. On the other hand, since at the moment no examples of scalar-flat K\"ahler AC non-Ricci flat models are known, we have not investigated this problem further.

\end{document}